\documentclass[letterpaper,12pt]{amsart}

\usepackage{amssymb}
\usepackage{amsmath,amscd}
\usepackage{rotating}

\numberwithin{equation}{section}

\newtheorem{lemma}{Lemma}[section]
\newtheorem{theorem}[lemma]{Theorem}
\newtheorem{proposition}[lemma]{Proposition}

\newtheorem{corollary}[lemma]{Corollary}

\theoremstyle{definition}

\newtheorem{remark}[lemma]{Remark}

\newcommand{\Gal}{\mathrm{Gal}}
\newcommand{\Hol}{\mathrm{Hol}}

\newcommand{\Aut}{\mathrm{Aut}}
\newcommand{\Perm}{\mathrm{Perm}}

\newcommand{\ord}{\mathrm{ord}}

\newcommand{\Z}{\mathbb{Z}}

\newcommand{\K}{\mathcal{K}}
\newcommand{\F}{\mathcal{F}}
\newcommand{\NN}{\mathcal{N}}

\newcommand{\LRA}{\Leftrightarrow}

\begin{document}
\title{Skew Braces of Squarefree Order}

\author{Ali A.~Alabdali}


\address{(A.~Alabdali) Department of Mathematics, College of Education for Pure Science,
University of Mosul, Mosul, Iraq.}   
\email{aaab201@exeter.ac.uk}

\author{Nigel P.~Byott}
\address{(N.~Byott) Department of Mathematics, College of Engineering,
  Mathematics and Physical Sciences, University of Exeter, Exeter 
EX4 4QF U.K.}  
\email{N.P.Byott@exeter.ac.uk}

\date{\today}

\bibliographystyle{amsalpha}

\begin{abstract} 
Let $n \geq 1$ be a squarefree integer, and let $M$, $A$ be two groups
of order $n$. Using our previous results on the enumeration of
Hopf-Galois structures on Galois extensions of fields of squarefree
degree, we determine the number of skew braces (up to isomorphism)
with multiplicative group $M$ and additive group $A$.  As an
application, we enumerate skew braces whose order is the product of three
distinct primes.
\end{abstract}

\maketitle

\section{Introduction}

Since the seminal work of Drinfeld \cite{Dr} on quantum groups, there
has been a great deal of interest in algebraic systems which give rise
to set-theoretic solutions of the quantum Yang-Baxter equation
(QYBE). Etingof, Schedler and Soloviev \cite{ESS} defined the
structure group attached to a set-theoretic solution of QYBE, and
studied many of its properties. Subsequently, Rump showed how
set-theoretic solutions arise from cycle sets \cite{Rump05} and
radical rings \cite{Rump}, and in the latter paper he also introduced
braces as a generalisation of radical rings. Braces give rise
nondegenerate involutive set-theoretic solutions of QYBE, and have
recently been studied intensively \cite{Bach-p3, CJO16, LV16,
  CGS18}. Several generalisatons of braces have been investigated,
including skew braces \cite{GV}, which give noninvolutive solutions to
QYBE, and semi-braces \cite{CCS}, which give degenerate solutions.

Another area of algebra in which radical rings have found application
is the study of Hopf-Galois structures on field extensions. If
$L/K$ is a finite Galois extension of fields with Galois group
$\Gamma$, then $L/K$ may admit many Hopf-Galois structures, and these can be 
described in group-theoretic
terms. In particular, the Hopf algebra $H$ acting on $L/K$ in any
Hopf-Galois structure is a twisted form of the group algebra $K[G]$,
where $G$ is a group acting regularly on $\Gamma$. The Hopf-Galois
structures can then be partitioned according to the isomorphism type
of $G$. We refer to the isomorphism type of $G$ as the {\em type} of the
Hopf-Galois structure. Using the connection found in \cite{CDVS} between radical
rings and abelian regular subgroups of the affine group of a field,
Featherstonhaugh, Caranti and Childs \cite{FCC} studied the possible
abelian types of Hopf-Galois structures on an abelian extension of
prime-power degree.

Given the role of radical rings in these two situations, it is perhaps
not surprising that there should be a deep connection between braces
and Hopf-Galois structures. This connection was first mentioned
explicitly in \cite{bachiller}, and was clarified and extended to the
setting of skew braces in the appendix to \cite{SV}. Several papers
and preprints have exploited this connection \cite{Childs-JA,
  Childs-arXiv, NZ}. 

In this paper, we make further use of the connection between skew
braces and Hopf-Galois structures in order to study skew braces 
of squarefree order. Given a squarefree number $n \geq 1$ and two
groups $M$, $A$ of order $n$, we determine the number $b(M,A)$ of
skew braces $(B,+,\ast)$ (up to isomorphism) with multiplicative group
$(B,\ast)$ isomorphic to $M$ and additive group $(B,+)$
isomorphic to $A$. To do so, we build upon our work in  
\cite{AB-galois} where, for any two groups $\Gamma$, $G$ of
squarefree order $n$, we determined the number $e(\Gamma,G)$ of
Hopf-Galois structures of type $G$ on a Galois extension with Galois
group $\Gamma$. The special case where $\Gamma$ is cyclic was
previously treated in \cite{AB}. In the final two sections of this
paper, we give some examples, in particular treating in full the case
where $n$ is the product of three primes. 

For comparison with our results, we note that the total number of skew
braces of order $n$ for each $n \leq 30$ is given in \cite[Table
  5.1]{GV}. These values were found using a computer calculation. The
braces (but not skew braces) with finite cyclic multiplicative group
are determined in \cite{Rump-cyc}, and the braces of order $p^3$ for $p$
prime are classified in \cite{Bach-p3}. The skew braces whose
multiplicative group is the Heisenberg group of order $p^3$ for $p>3$
are enumerated in \cite{NZ}. Further results on skew braces of order
$p^3$ can be found in \cite{NZ-thesis}.

\section{Statement of Main Result}

Before stating our main result, we must first describe the groups of
squarefree order.  Specialising the characterisation in \cite{MM} of
finite groups in which every Sylow subgroup is cyclic, we obtain the
following classification, cf.~\cite[Lemma 3.2]{AB}.

\begin{lemma} \label{sf-class}
Let $n \geq 1$ be squarefree. Then any group of order $n$ has the form
$$   G(d,e,k)= \langle \sigma, \tau \colon \sigma^e=\tau^d=1, \tau
\sigma \tau^{-1} = \sigma^k \rangle $$ 
where $n=de$, $\gcd(d,e)=1$ and $\ord_e(k)=d$. Conversely, any choice
of $d$, $e$ and $k$ satisfying these conditions gives a group
$G(d,e,k)$ of order $n$. Moreover, two such groups $G(d,e,k)$ and
$G(d',e',k')$ are isomorphic if and only if $d=d'$, $e=e'$, and
$k$, $k'$  generate the same cyclic subgroup of $\Z_e^\times$.
\end{lemma}

Here, for a natural number $m$, we write $\Z_m$ for the ring of integers modulo
$m$, and $\Z_m^\times$ for its group of units. Also, for $a \in \Z$
with $\gcd(a,m)=1$, we denote by $\ord_m(a)$ the order of $a$ in $\Z_m^\times$. 

We now fix a squarefree number $n$, and two groups of order
$n$:
\begin{equation} \label{def-AM} 
  A=G(d,e,k), \qquad M=G(\delta, \epsilon,\kappa). 
\end{equation}
(Our notation is chosen to be consistent with \cite{AB-galois}, except
that $A$, $M$ correspond respectively to $G$, $\Gamma$ in that paper.)
We then define 
\begin{equation} \label{zg-gamma-zeta}  
    z=\gcd(k-1, e),   \quad g=e/z, \quad  \zeta=\gcd(\kappa-1,
    \epsilon),   \quad 
     \gamma=\epsilon/\zeta. 
\end{equation}
Thus $z$, $g$ depend only on $A$, and $\zeta$, $\gamma$ depend only on
$M$. We also set
$$  w=\varphi(\gcd(\delta,d)), $$
which depends on both $A$ and $M$. Here $\varphi$ is the Euler totient
function.  

We can now state our main result. Recall that $\omega(g)$ denotes
the number of (distinct) prime factors of the squarefree integer $g$. 

\begin{theorem} \label{thm-braces}
Let $M$ and $A$ be groups of squarefree order $n$. Then, with the
above notation, the number $b(M,A)$ of isomorphism classes of skew
braces with multiplicative group isomorphic to $M$ and additive group
isomorphic to $A$ is given by
$$ b(M,A) = \begin{cases} 2^{\omega(g)} w & \mbox{if } \gamma \mid e,
  \\ 0 & \mbox{if } \gamma \nmid e. \end{cases} $$
\end{theorem}

\section{Skew Braces and Hopf-Galois Structures} \label{SB-HGS}

For the convenience of the reader, we review in this section the
necessary background on skew braces and Hopf-Galois structures,
emphasising the connection between the corresponding enumeration
problems.

A skew left brace $(B,+,\ast)$ is a set $B$ with two binary operations
such that $(B,+)$ and $(B,\ast)$ are groups, and $a\ast(b+c) = (a \ast
b) + (-a) + (a \ast c)$ for all $a$, $b$, $c \in B$, where $-a$ is the
inverse of $a$ under $+$.  We call $(B,+)$ the additive group of $B$
and $(B,\ast)$ the multiplicative group. If $(B,+)$ is abelian, then
$(B,+,\ast)$ is a left brace. Right skew braces and right braces are
defined analogously, but we will not need these concepts in this
paper. We therefore omit the adjective ``left'' from now on.

An isomorphism between two skew braces is a bijection between their
underlying sets which is an isomorphism of both the additive groups
and the multiplicative groups. For groups $M$, $A$ of the same order, we write $b(M,A)$ for the number of
isomorphism types of skew braces $(B,+,\ast)$ with $(B,\ast) \cong M$
and $(B,+) \cong A$.

Let $(B,+,\ast)$ be a skew brace. There is a homomorphism of groups 
$$ \lambda:(B,\ast) \to \Aut(B,+), \qquad b \mapsto \lambda_b \mbox{
  with }\lambda_b(a) = b \ast a - a. $$
Thus $(B,\ast)$ acts on $(B,+)$. Moreover, the identity map on $B$
induces a bijection $i:(B,\ast) \to (B,+)$ which satisfies the
$1$-cocycle identity with respect to this action:
$$ i(b \ast c) = i(b) + \lambda_b(i(c)) \mbox{ for all } b, c \in B.  $$

Now consider the holomorph $\Hol(B,+)=(B,+) \rtimes \Aut(B,+)$ of the
group $(B,+)$. We view $\Hol(B,+)$ as a subgroup of the group $\Perm(B)$
of permutations of the underlying set $B$, with the normal subgroup
$(B,+)$ of $\Hol(B,+)$ acting as left translations. The map $(i,\lambda) : (B,\ast) \to
\Hol(B,+)$ is a group homomorphism whose image is regular on
$B$. (Recall that a subgroup $G \subset \Perm(X)$ is regular on $X$ if
it is transitive and the stabiliser of any point in $X$ is the trivial
group. When $G$ is finite, this implies that $|G|=|X|$.) Thus the skew
brace $(B,+,\ast)$ gives rise to a regular embedding of $(B,\ast)$
into the holomorph of $(B,+)$. 

We can reverse this construction: for abstract finite groups $M$ and
$A$, a regular embedding $\beta: M \to \Hol(A)$ gives rise to a
homomorphism $\lambda: M \to \Aut(A)$, and so to an action of $M$ on
$A$, together with a bijective cocycle $i:M \to A$ with respect to
this action. Letting $(B,\ast)=M$, and defining a new operation $+$ on
$B$ by pulling back the group operation of $A$ via $i$, we obtain a
skew brace $(B,+,\ast)$ with $(B,\ast) \cong M$ and $(B,+) \cong
A$. We may compose $\beta$ (on the right) with any element of
$\Aut(M)$, and this does not change the isomorphism type of $B$. An
automorphism $\psi$ of $A$ induces an automorphism of $\Hol(A)$,
which is given by conjugation with $\psi$ in $\Hol(A)$. Composing
$\beta$ (on the left) with this automorphism again does not change
the isomorphism type of $B$. Thus $b(M,A)$ is just the number of
$\Aut(A) \times \Aut(M)$-orbits of regular embeddings $M \to \Hol(A)$,
where the action of $\Aut(M)$ on regular embeddings is by composition,
and that of $\Aut(A)$ by conjugation. Since two regular embeddings
$\beta$ have the same image if and only if they are in the same
$\Aut(M)$-orbit, $b(M,A)$ is also number of $\Aut(A)$-orbits of
regular subgroups in $\Hol(A)$ isomorphic to $M$.

We note that the groups $\Aut(A)$ and $\Aut(M)$ each act 
without fixed points on the set of regular embeddings $\beta : M \to \Hol(A)$, 
but their product $\Aut(A) \times \Aut(M)$ does not. Thus the orbits of $\Aut(A) \times \Aut(M)$ on this set of regular embeddings, or, equivalently, the orbits of 
$\Aut(A)$ on the set of regular subgroups of $\Hol(A)$ isomorphic to $M$, are in general not all of the same size. 

We now turn to Hopf-Galois structures. A Hopf-Galois structure on a
finite extension of fields $L/K$ consists of a cocommutative $K$-Hopf
algebra $H$ and an action of $H$ on $L$ making $L/K$ into an
$H$-Galois extension in the sense of Chase and Sweedler \cite{CS}. The
motivating example is when $L/K$ is a Galois extension in the
classical sense, and $H$ is the group algebra $H=K[\Gamma]$ for
$\Gamma=\Gal(L/K)$, with the natural action of $H$ on $L$. The Galois
extension $L/K$ may also admit other (non-classical) Hopf-Galois
structures. Greither and Pareigis \cite{GP} showed that the
Hopf-Galois structures on $L/K$ correspond bijectively to the regular
subgroups $G \subset \Perm(\Gamma)$ which are normalised by the group
$\lambda(\Gamma)$ of left translations by $\Gamma$. This shows that
the number of such Hopf-Galois structures can be obtained by a purely
group-theoretic calculation. The Hopf algebra occurring in the
Hopf-Galois structure given by the regular subgroup $G$ is
$L[G]^\Gamma$, the Hopf algebra of $\Gamma$-fixed points in the group
algebra $L[G[$, where $\Gamma$ acts simultaneously on $L$ as field
automorphisms and on $G \subset \Perm(\Gamma)$ as conjugation with
left translations. We refer to the isomorphism type of $G$ as the {\em
  type} of this Hopf-Galois structure. We write $e(\Gamma,G)$ for the
number of Hopf-Galois structures of type $G$ on a Galois extension
$L/$K with $\Gal(L/K) \cong \Gamma$.

We change notation to facilitate comparison between the
enumeration problems for skew braces and for Hopf-Galois
structures. Let $M$ and $A$ be two abstract finite groups of the same
order. Then $e(M,A)$ is the number of regular subgroups in $\Perm(M)$
isomorphic to $A$ and normalised by $\lambda(M)$. Equivalently,
$e(M,A)$ is the number of $\Aut(A)$-orbits of regular embeddings
$\alpha:A \to \Perm(M)$ with image normalised by $\lambda(M)$. There
is a canonical bijection between regular embeddings $\alpha:A \to
\Perm(M)$ and regular embeddings $\beta : M \to \Perm(A)$, under which
$\alpha(A)$ is normalised by $\lambda(M)$ if and only if $\beta(M)
\subset \Hol(A)$. Thus $e(M,A)$ is the number of $\Aut(A)$-orbits of
regular embeddings $\beta: M \to \Hol(A)$. Since the $\Aut(M)$-orbits
of such $\beta$ correspond to the regular subgroups of $\Hol(A)$
isomorphic to $M$, we may in practice evaluate $e(M,A)$ by first
determining the number $e'(M,A)$ of these subgroups. We then have
$$ e(M,A) = \frac{|\Aut(M)|}{|\Aut(A)|} \ e'(M,A). $$ 
This is often simpler than working directly with regular subgroups of
$\Perm(M)$, as $\Hol(A)$ is usually much smaller than $\Perm(M)$.

In summary, $b(M,A)$ is the number of $\Aut(A)$-orbits on the set
of regular subgroups of $\Hol(A)$ isomorphic to $M$, whereas $e'(M,A)$ 
is just the number of such subgroups. Once we have determined all such subgroups, 
we can easily find the number $e(M,A)$ of Hopf-Galois structures, but, as the $\Aut(A)$-orbits 
on these subgroups
can be of different sizes, obtaining the number $b(M,A)$ of skew braces  
may require significant further calculation.

\begin{remark}
For Galois extensions of squarefree degree, the formula for
$e(\Gamma,G)$ we obtained in \cite[Theorem 2.2]{AB-galois} is much
more complex than the formula for the number of skew braces in Theorem
\ref{thm-braces} of this paper. This is because the
number of Hopf-Galois structures depends in an intricate way on
the interplay of the structures of the two groups $\Gamma$, $G$. We
are not aware, however, of any approach to proving Theorem
\ref{thm-braces} of this paper which avoids the detailed analysis of
regular subgroups at the heart of our enumeration of the Hopf-Galois structures.
\end{remark}

\section{Regular subgroups in $\Hol(A)$}

Let $M$ and $A$ be two groups of squarefree order $n$, as in Theorem
\ref{thm-braces}.  In this section, we recall some results from
\cite{AB-galois}, and in particular we determine in Lemma \ref{quin}
explicit generators for all regular subgroups of $\Hol(A)$ isomorphic
to $M$. We will keep to the notation of \cite{AB-galois} except that
we write $A$, $M$ in place of $G$, $\Gamma$.

From (\ref{def-AM}) and Lemma \ref{sf-class}, we have the presentations 
 $$ A = \langle \sigma, \tau \colon \sigma^e=\tau^d=1, \tau
\sigma \tau^{-1} = \sigma^k \rangle,   $$  
$$ M = \langle s, t \colon s^\epsilon=t^\delta=1, 
 tst^{-1} = s^\kappa \rangle.  $$

Recall that $g$, $z$, $\gamma$, $\zeta$ were defined in
(\ref{zg-gamma-zeta}).

By \cite[Lemma 4.1]{AB}, $\Aut(A) \cong \Z_g \rtimes \Z_e^\times$ and
is generated by the automorphism $\theta$, and automorphisms $\phi_s$ 
for $s \in \Z_e^\times$ where
$$ \theta(\sigma) = \sigma, \quad \theta(\tau) = \sigma^z \tau; \qquad 
  \phi_s(\sigma) = \sigma^s, \quad \phi_s(\tau)= \tau. $$

We write elements of $\Hol(A)=A \rtimes \Aut(A)$ as $[x,\alpha]$,
where $x \in A$ and $\alpha \in \Aut(A)$. An arbitrary element of
$\Hol(A)$ therefore has the form $[\sigma^u \tau^f, \theta^v \phi_t]$
for $u \in \Z_g$, $f \in \Z_d$, $v \in \Z_g$ and $t \in \Z_e^\times$.
The group operation in $\Hol(A)$, and the action of $\Hol(A)$ on $A$,
are given by
\begin{equation} \label{mul-Hol}
    [x, \alpha] [y, \beta] = [x \alpha(y), \alpha \beta], 
   \qquad [x, \alpha] \cdot y = x\alpha(y).
\end{equation}
We modify the above presentation of $M$. Setting $X=s^\zeta$ and
$Y=ts^\gamma$, we have the alternative presentation
\begin{equation} \label{new-pres}
   \Gamma = \langle X, Y \colon X^\gamma=Y^{\zeta\delta} =1, YXY^{-1}
   = X^\kappa \rangle.
\end{equation}

\begin{proposition}{\cite[Proposition 5.2]{AB-galois}} \label{d-div-gam-del}
If $\Hol(A)$ contains a regular subgroup $M^* \cong M$ then $\gamma
\mid e$. Moreover, if $X$ and $Y$ are generators of $M^*$ satisfying
the relations in (\ref{new-pres}) then the subgroup $\langle X, Y^d
\rangle$ of $M$ of order $e$ acts regularly on the subset  $\{
\sigma^m : m \in \Z\}$ of $A$.
\end{proposition}

If $\gamma \nmid e$ then Proposition \ref{d-div-gam-del}, together
with the discussion in \S\ref{SB-HGS}, shows that there are no skew
braces $(B,+,\ast)$ with $(B,\ast) \cong M$ and $(B,+) \cong A$. This
already proves the second case of Theorem \ref{thm-braces}.

We can replace $Y$ by some power
$Y^f$ to ensure that $Y$ has the form $Y=[\sigma^u \tau, \theta^v
  \phi_t]$ (so $\tau$ occurs in $Y$ with exponent $1$), but in doing
so we replace $\kappa$ by $\kappa^f$. To allow for this. we consider
the action of the group
$$ \Delta:=\{m \in \Z_\delta^\times : m \equiv 1 \pmod{\gcd(\delta,d)}\} $$
on the set
$$  \K = \{ \kappa^r : r \in \Z_\delta^\times\}.  $$
The index of $\Delta$ in $\Z_\delta^\times$ is
$w=\varphi(\gcd(\delta,d))$. We choose a system $\kappa_1, \ldots,
\kappa_w$ of orbit representatives of $\Delta$ on $\K$. 

\begin{lemma}{\cite[Lemma 5.5]{AB-galois}} \label{h unique}
Let $M^* \cong M$ be a regular subgroup of $\Hol(A)$. 
Then there is a unique $h$ with $1 \leq h \leq w$ such
that $M^* $ is generated by a pair of elements $X$, $Y$ of the form 
\begin{equation} \label{def-XY}
  X=[ \sigma^a, \theta^c], \qquad Y=[\sigma^u \tau,
    \theta^v \phi_{t}] .
\end{equation}
which satisfy the relations
\begin{equation} \label{XYkap-h}
X^\gamma = Y^{\zeta\delta} = 1, \qquad YXY^{-1} = X^{\kappa_h}.
\end{equation}
Indeed, $M^*$ contains exactly $\gamma \varphi(e) w /
\varphi(\delta)$ such pairs of generators.
\end{lemma} 

The relations in (\ref{XYkap-h}) are the same as those in
(\ref{new-pres}), except that $\kappa$ is replaced by one of the coset
representatives $\kappa_h$.
Lemma \ref{h unique} shows that the regular subgroups of $\Hol(A)$
isomorphic to $M$ fall into $w$ disjoint families $\F_1, \ldots,
\F_w$ corresponding to the orbits of $\Delta$ on $\K$. 

Not every pair $X$, $Y$ satisfying (\ref{def-XY}) and (\ref{XYkap-h})
generate a regular subgroup. To identify those which do, we let
$$ \NN_h  \subset \Z_e^\times \times \Z_e \times \Z_g \times \Z_e \times \Z_g$$  
be the set of quintuples  $(t, a,c,u,v)$
such that the corresponding elements  $X$, $Y$ satisfy  (\ref{XYkap-h}) for the orbit representative
$\kappa_h$ and generate a {\em regular} subgroup 
$M^*=\langle X,Y\rangle \in \F_h$.

It follows from Lemma \ref{h unique} that the number $e'(M,A)$ of
regular subgroups of $\Hol(A)$ isomorphic to $M$ is given by 
$$ e'(M,A) = \sum_{h=1}^w |\F_h| =   \frac{\varphi(\delta)}{\gamma
\varphi(e)w} \sum_{h=1}^w |\NN_h|. $$

We recall from \cite[Lemma 4.3]{AB-galois} a formula for powers of the
element $Y$.

\begin{lemma}  \label{Y-pow}
For $Y$ as in (\ref{def-XY}) and $j\geq 0$, we have 
$$   Y^j = [ \sigma^{A(j)} \tau^j, \theta^{vS(t,j)} \phi_{t^j}], $$ 
where $A(j) = uS(tk,j) + vzkT(k,t,j)$, with 
\begin{equation}  \label{def-S-sum}
     S(m,j)= \sum_{i=0}^{j-1} m^i. 
\end{equation}
and
\begin{equation} \label{def-T-sum}
  T(k,t,j) = \sum_{h=0}^{j-1} S(t,h) k^{h-1} \mbox{ for } j \geq 1,
  \qquad  T(k,t,0)=0.
\end{equation}
\end{lemma}

 For each prime $q \mid e$ (respectively, $q \mid \epsilon$), we set
 $r_q = \ord_q(k)$ (respectively, $\rho_q=\ord_\epsilon(\kappa)$). We
 then divide the set of primes $q \mid e$ into six subsets (any of
 which may be empty) as follows.
 $$ P = \{\mbox{primes }q \mid \gcd(\gamma,z)  \}; $$
$$  Q = \{\mbox{primes }q \mid \gcd(\zeta \delta,z)  \}; $$
$$  R= \{\mbox{primes }q \mid \gcd(\gamma,g) : \rho_q \neq r_q \}; $$
$$  S =  \{\mbox{primes }q \mid \gcd(\gamma,g) : \rho_q=r_q >2 \}; $$ 
$$  T = \{\mbox{primes }q \mid \gcd(\gamma,g) : \rho_q=r_q=2 \}; $$
$$  U = \{\mbox{primes }q \mid \gcd(\zeta \delta,g)  \}. $$
Moreover, for each $h \in \{1, \ldots, w\}$, we define
$$   S_h^+ = \{ q \in S :  \kappa_h \equiv k \pmod{q} \}, $$
$$  S_h^- = \{ q \in S : \kappa_h \equiv k^{-1} \pmod{q} \}, $$
and set
$$   S_h = S_h^+ \cup S_h^-, \qquad  S'_h = S \backslash S_h. $$ 
We also set
$$  \lambda = z^{-1}(k-1) \in \Z_g^\times.  \qquad 
  \mu = k^{-1} z^{-1}(k-1) \in \Z_g^\times. $$

The following result is \cite[Lemma 6.12]{AB-galois}.

\begin{lemma} \label{quin}
A quintuple $(t,a,c,u,v) \in \Z_e^\times \times \Z_e \times \Z_g \times
  \Z_e \times \Z_g$ belongs to $\NN_h$ if and only if, for each prime
  $q \mid e$, its entries satisfy the conditions mod $q$ shown in
  Table \ref{quintuples}.
\end{lemma}

\begin{table}[ht] 
\centerline{ 
\begin{tabular}{|c|c|c|c|c|c|} \hline
  Primes $q$ & $t$ & $a$ & $u$ & $c$ & $v$   \\ \hline 
$ q \in P $ & $ \kappa $ & $ \not \equiv 0 $
  & arb.  & &  \\ \hline  
$ q\in Q $ & $ 1 $ & $0 $ & $ \not \equiv 0$ & & \\  \hline
 $q \in R \cup S_h'$ & $ \kappa $ & $  \not \equiv 0 $ & arb. 
              & $\lambda a$ & arb. \\   
 & $\kappa k^{-1}$ & $ \not \equiv 0 $ &
                        arb. & $ 0 $ & arb.  \\ \hline
$q \in S_h^+$  & $ \kappa k^{-1} \equiv 1
             $ & $ \not \equiv 0 $ & arb. & $ 0 $ & $ 0 $   \\ 
& $\kappa $ & $\not \equiv
                        0 $ & arb. &   $\lambda a$   & arb.  \\  \hline 
$q \in S_h^-$ & $ \kappa $ & $ \not \equiv 0 $
                        & arb.  & $\lambda a$ & $\mu u$ \\  
	 & $\kappa k^{-1} \equiv \kappa^2$ & $\not \equiv 0$ & arb. &
                        $ 0 $ & arb.  \\ \hline  
 $q \in T$ & $\kappa \equiv -1$ & $\not \equiv 0$ 
                   & arb. & $\lambda a$ & $\mu u$  \\    
 & $\kappa k^{-1} \equiv 1$ & $ \not
                        \equiv0 $ & arb. & $0$  & $ 0 $ \\ \hline   
$ q \in U$ & $1$ & $0$ & arb. & $ 0
                        $ & $ \not \equiv 0$   \\  
	& $ k^{-1}$ & $0$ & arb. & $0$ & $\not \equiv \mu u$ 
                        \\ \hline 
\end{tabular}
}  
\vskip3mm

\caption{Conditions for membership of $\NN_h$.} 
 \label{quintuples}  	
\end{table}

In Table \ref{quintuples} and the discussion below, we write $\kappa$ in place of $\kappa_h$ to simplify notation. Moreover, all congruences are modulo the relevant prime $q$ unless otherwise indicated. The entries in Table \ref{quintuples} for $c$, $v$ are blank when $q \mid z$ (so $q \in P \cup
Q$) since $c$, $v$ are only defined mod $g$.

To illustrate how to interpret Table \ref{quintuples}, suppose that $q
\in S_h^+$, so that $\kappa \equiv k$. There are two possibilities for
$t \bmod q$. Either $t \equiv \kappa$, and then $a \not \equiv 0$, $c
\equiv \lambda a$ and $u$ and $v$ maybe chosen arbitrarily mod $q$,
or else $t \equiv 1$, and then $u$ may be chosen arbitrarily but $c
\equiv v \equiv 0$.

Finally, from \cite[Corollary 6.10, Proposition 6.11]{AB-galois} we have the
following congruence information on the exponents $A(j)$ in Lemma
\ref{Y-pow} when $j$ is a multiple of $d$:

\begin{lemma} \label{Adi} 
For $i \geq 0$, we have the following congruences mod $q$.
\begin{itemize}
\item[(i)] If $q \in Q$ then $A(di) \equiv udi$.
\item[(ii)] If $q \in U$ with $t \equiv 1$ then 
$$ A(di) \equiv \frac{vzdi}{k-1}. $$
\item[(iii)] If $q \in U$ with $t \equiv k^{-1}$ then 
$$ A(di) \equiv \frac{zk}{k-1}(v -\mu u) di. $$
\item[(iv)] If $q \mid \gamma$ then $A(di) \equiv 0$ for all $i \equiv
  0 \pmod{\gcd(\delta,e)}$.
\end{itemize}
\end{lemma}

\section{Counting Skew Braces}

Lemma \ref{quin} enables us to find all regular subgroups of $\Hol(A)$
isomorphic to $M$. Indeed, given $h \in \{1, \ldots w\}$ and a
quintuple $(t,a,c,u,v) \in \NN_h$, we have the corresponding regular
subgroup $M^*=\langle X, Y \rangle$ where $X$, $Y$ are defined in
(\ref{def-XY}). By Lemma \ref{h unique}, every regular subgroup $M^*$
arises this way from $\gamma \varphi(e)w/\varphi(\delta)$ quintuples
in $\NN_h$. The number $b(M,A)$ of skew braces with multiplicative
group $M$ and additive group $A$ is the number of $\Aut(A)$-orbits of
these regular subgroups, where $\Aut(A)$ acts on $\Hol(A)$ by
conjugation.  

Let $I_h(t,a,c,u,v)$ be the size of the orbit of the
subgroup $M^*$. Then $I_h(t,a,c,u,v)$ is the index in $\Aut(A)$ of the
stabiliser of $M^*$. To count the skew braces, we need to count the
subgroup $M^*$ with weight $I_h(t,a,c,u,v)^{-1}$. Thus we have the
following formula for $b(M,A)$:
\begin{equation} \label{count-braces}
  b(M,A) = \frac{\varphi(\delta)}{\gamma \varphi(e) w}
   \sum_{h=1}^w \sum_{(t,a,c,u,v) \in \NN_h} \frac{1}{I_h(t,a,c,u,v)}. 
\end{equation}  

It remains to determine $I_h(t,a,c,u,v)$. An element of
$\Aut(A)$ has the form $\alpha = \theta^r \phi_s$ for $r \in \Z_g$ and $s
\in \Z_e^\times$. 

\begin{lemma}
Suppose that $\gamma \mid e$, let $1 \leq h \leq w$, and let 
$M^* = \langle X,Y \rangle$ be the regular subgroup in $\Hol(A)$
corresponding to $(t,a,c,u,v) \in \NN_h$ as above. Let $\alpha =
\theta^r \phi_s \in \Aut(A)$. Then $\alpha M^* \alpha^{-1}=M^*$ in
$\Hol(A)$ if and only if there exist $i \in \Z_{\zeta \delta/d}$ and
$j \in \Z_ \gamma$ satisfying the following three conditions:
\begin{equation} \label{t-di}
   t^{di} \equiv 1 \pmod{e}; 
\end{equation}
\begin{equation} \label{cong2}
   cj+vS(t,di)  \equiv  v(s-1)-(t-1)r  \pmod{g};
\end{equation}
\begin{equation} \label{cong3}
   aj+A(di) \equiv (u-zv)(s-1)+ztr \pmod{e}.
\end{equation}
Moreover, (\ref{t-di}) is equivalent to 
\begin{equation} \label{i-lin}
  i \equiv 0 \pmod{\gcd(\delta,e)}; 
\end{equation}

\end{lemma}
\begin{proof}
Using (\ref{mul-Hol}) to calculate in $\Hol(A)$, we have
\begin{eqnarray*}
 \alpha X \alpha^{-1} & = & [1, \theta^r \phi_s] [\sigma^a,\theta^c] [
   1, \phi_s^{-1} \theta^{-r}] \\ 
  & = & [\sigma^{as} , \theta^r \phi_s \theta^c \phi_s^{-1} \theta^{-r}
 ] \\
 & = & [\sigma^{as}, \theta^{r+cs-r}] \\
 & = & X^s, 
\end{eqnarray*}
and
\begin{eqnarray*}
 \alpha Y \alpha^{-1} & = & [1, \theta^r \phi_s] [\sigma^u \tau,
   \theta^v \phi_t] [ 1, \phi_s^{-1} \theta^{-r}] \\  
   & = & [\sigma^{us+rz} \tau, \theta^r \phi_s \theta^v \phi_t
   \phi_s^{-1} \theta^{-r}] \\
 & = & [\sigma^{us+rz} \tau, \theta^{r+vs}\phi_t \theta^{-r}] \\
 & = & [\sigma^{us+rz} \tau, \theta^{r+vs-rt}\phi_t].
\end{eqnarray*}
Also 
$$ Y^{-1} = [ \phi_t^{-1} \theta^{-v}(\tau^{-1} \sigma^{-u}),
  \phi_t^{-1} \theta^{-v}], $$
so that 
\begin{eqnarray*} 
 (\alpha Y \alpha^{-1})Y^{-1} & = & 
  [\sigma^{us+rz} \tau, \theta^{r+vs-rt}\phi_t] 
 [ \phi_t^{-1} \theta^{-v}(\tau^{-1} \sigma^{-u}),
  \phi_t^{-1} \theta^{-v}] \\
  & = & [\sigma^{us+rz} \tau  \cdot  \theta^{r+vs-rt-v}(\tau^{-1}
   \sigma^{-u}), \theta^{r+vs-rt-v}]  \\
  & = & [\sigma^{us+rz} \tau  \cdot \tau^{-1} \sigma^{-z(r+vs-rt-v)-u}, 
            \theta^{r+vs-rt-v}]  
\end{eqnarray*}
Thus
\begin{equation} \label{aycomm}
  (\alpha Y \alpha^{-1})Y^{-1}  =  
    [\sigma^{(u-zv)(s-1)+ztr]}, \theta^{v(s-1)-r(t-1)}]. 
\end{equation}
We then have 
\begin{eqnarray*}
 \alpha M^* \alpha^{-1}=M^* & \LRA & \langle X^s, \alpha Y
 \alpha^{-1} \rangle = \langle X, Y \rangle \\
  & \LRA & \langle X, \alpha Y  \alpha^{-1} \rangle = \langle X, Y \rangle \\
  & \LRA & (\alpha Y  \alpha^{-1}) Y^{-1} \in \langle X, Y \rangle.
\end{eqnarray*}
But $(\alpha Y \alpha^{-1})Y^{-1}$ does not involve $\tau$, so this
element takes $1_A$ to $\sigma^m$ for some $\sigma \in \Z$. Since
$M^*$ is regular on $A$, it follows from Proposition
\ref{d-div-gam-del} that
$$  \alpha M^* \alpha^{-1}=M^*  \LRA 
   (\alpha Y  \alpha^{-1}) Y^{-1} \in \langle X, Y^d \rangle. $$
Thus $\alpha M^* \alpha^{-1} = M^*$ if and only if there exist
 $i \in \Z_{\zeta \delta/d}$ and $j \in \Z_\gamma$ such that
$$  (\alpha Y \alpha^{-1})Y^{-1}  = X^j Y^{di} =
[\sigma^{aj+A(di)}, \theta^{cj+vS(t,di)} \phi_{t^{di}}], $$ 
where the last equality comes from Lemma \ref{Y-pow}.  
Comparing with (\ref{aycomm}), this is equivalent to (\ref{t-di}),
(\ref{cong2}) and (\ref{cong3}).

We show that (\ref{t-di}) is equivalent to (\ref{i-lin}).  Suppose
(\ref{t-di}) holds. Since $k^d \equiv 1 \pmod{e}$, it follows from the
values of $t$ shown in Table \ref{quintuples} that $\kappa^{di} \equiv
1 \pmod{\gamma}$. As $\kappa \equiv 1 \pmod{\zeta}$, we then have
$\kappa^{di} \equiv 1 \pmod{\epsilon}$, so $\delta \mid di$. Thus
$p \mid i$ for every prime $p \mid \delta$ with $p \nmid d$. This
implies (\ref{i-lin}).  Conversely, if (\ref{i-lin}) holds then
$\kappa^{di} \equiv 1 \pmod{\epsilon}$. Since also $k^{di} \equiv 1
\pmod{e}$, it follows from Table \ref{quintuples} that (\ref{t-di})
holds.
\end{proof}

To calculate $I_h(t,a,c,u,v)$, we must find the proportion of pairs
$r$, $s$ for which (\ref{cong2})--(\ref{i-lin}) can be solved for $i$
and $j$. If (\ref{i-lin}) holds then, at each prime $q \mid e$,
(\ref{cong2}) and (\ref{cong3}) reduce to congruences which are linear
in $i$ (as well as in $j$, $r$ and $s$). Indeed, using Lemma
\ref{Adi}, we can replace $A(di)$ by a multiple of $di$ (by $0$
if $q \mid \gamma$). Also, as $t^{di} \equiv 1$, we have $S(t,di)
\equiv 0$ if $t \not \equiv 1$ and $S(t,di) \equiv di$ if $t \equiv
1$. By the Chinese Remainder Theorem, this linearity, together with
the obvious linearity of (\ref{i-lin}), means that solutions to
(\ref{cong2})--(\ref{i-lin}) exist if and only if solutions exist mod $q$
for each prime $q \mid e$, and that the existence of solutions mod $q$
depends only on the residue classes of $r$, $s$ mod $q$. Thus we can
decompose $I_h(t,a,c,u,v)$ into contributions for each $q$:
$$ I_h(t,a,c,u,v) = \prod_{q \mid e} I_q. $$
(We suppress the dependence of $I_q$ on $h$ and $(t,a,c,u,v)$ from the
notation.) 

In order to calculate the $I_q$, we need to subdivide the set of primes
$q \mid e$ more finely than we did in Lemma \ref{quin}. We define
$$ Q' = \{ \mbox{primes } q \mid \gcd(\delta,z) \}, $$
$$ Q'' = \{ \mbox{primes } q \mid \gcd(\zeta,z) \}, $$
$$ U' = \{\mbox{primes } q \mid \gcd(\delta,g) \}, $$
$$ U'' = \{\mbox{primes } q \mid \gcd(\zeta,g) \}, $$
$$ S_{h,1}^+= \{q \in S_h^+ : t \equiv 1 \}, $$
$$ S_{h,2}^+= \{q \in S_h^+ : t \equiv \kappa \}, $$
$$ S_{h,1}^-= \{q \in S_h^- : t \equiv \kappa \}, $$
$$ S_{h,2}^-= \{q \in S_h^- : t \equiv \kappa k^{-1} \}. $$
Thus we have (disjoint) unions

$$ Q = Q' \cup Q'', \qquad U=U' \cup U'', \qquad 
  S_h^+=S^+_{h,1} \cup S^+_{h,2}, \qquad 
  S_h^-=S^-_{h,1} \cup  S^-_{h,2}. $$

\begin{lemma}
Let $M^*$ correspond to $(t,a,c,u,v) \in \NN_h$ as before. For each
prime $q \mid e$, the $q$-part $I_q$ of the index of the stabiliser of
$M^*$ in $\Aut(A)$ is as shown in Table \ref{Iq}. (For ease of
reference, we repeat in Table \ref{Iq} the possible values of $t$,
$a$, $c$, $u$, $v$ mod $q$ as given in Table \ref{quintuples}. We
also show the number $N_q$ of such quintuples mod $q$.)
\end{lemma}

\begin{table}[ht] 
\centerline{ 
\begin{tabular}{|c|c|c|c|c|c|c|c|} \hline
 Primes $q$ & $t$ & $a$ & $u$ & $c$ & $v$ & Index $I_q$ & Number $N_q$
 \\ \hline  
  $q \in P$ & $\kappa$ & $\not \equiv 0$ & arb. & & &
 $1$ & $q(q-1)$ \\ \hline
  $q \in Q'$ & $1$ & $0$ & $\not \equiv 0$ & & & $q-1$ &$q-1$ \\ \hline
   $q \in Q''$ & $1$ & $0$ & $\not \equiv 0$ & &  & $1$ &
 $q-1$ \\ \hline  
 $q \in R \cup S_h'$ & $\kappa$ & $\not \equiv 0$ & arb. &
$\lambda a$ & arb. & $q$ & $2q^2(q-1)$ \\   
       	& $\kappa k^{-1}$ & $\not \equiv 0$ & arb. & $0$ & arb. & &  \\ \hline
  $q \in S_{h,1}^+$ & $\kappa k^{-1} \equiv 1$ & $\not \equiv 0$ & arb. &
 $0$ & $0$ & $1$ & $q(q-1)$ \\ \hline 
 $q \in S_{h,2}^+$ & $\kappa$ & $\not \equiv 0$ &
 arb. & $\lambda a$ & arb. & $q$ & $q^2(q-1)$ \\ \hline
$q \in S_{h,1}^-$ & $\kappa$  & $\not \equiv 0$ &
 arb. &  $\lambda a$ & $\mu u$ & $1$  & $q(q-1)$ \\ \hline 
$q \in S_{h,2}^-$ & $\kappa k^{-1}$ & $ \not
 \equiv 0 $ & arb. & $ 0 $ &  arb. & $q$ & $q^2(q-1)$ \\ \hline   
 $q \in T$ 
  & $\kappa \equiv -1$ & $\not \equiv 0$ & arb. & $\lambda a$
 & $\mu a$ & $1$ & $2q(q-1)$ \\  
 & $\kappa k^{-1} \equiv 1$ & $\not \equiv 0$ & arb. & $0$ & $0$  &
   &  \\ \hline 
  $q \in U'$ & $1$ & $0$ & arb. & $0$ & $\not \equiv 0$ &
 $q(q-1)$ &  $2q(q-1)$ \\
 & $k^{-1}$ & $0$ & arb. & $0$ & $\not \equiv \mu u$ &  & \\ \hline 
$q \in U''$  & $1$ & $0$ & arb. & $0$ &
$\not \equiv 0$ & $q$ & $2q(q-1)$ \\ 
  & $k^{-1}$ & $0$ & arb. & $0$ & $\not \equiv \mu u$ &  & 
 \\ \hline  
\end{tabular}
}  
\vskip3mm

\caption{$q$-parts of index of stabiliser and number of quintuples.}
\label{Iq}
\end{table}

\begin{proof}
We suppose that (\ref{i-lin}) holds. Then, by Lemma \ref{Adi}(iv),
$A(di) \equiv 0 \pmod{q}$ for each $q \mid \gamma$. To find $I_q$
for a given $q$, we divide the number of pairs $r \in \Z_q$, $s \in
\Z_q^\times$ by the number of such pairs for which (\ref{cong2}),
(\ref{cong3}) can be solved mod $q$ for $i$, $j$ with $i$ also
satisfying (\ref{i-lin}). We omit any of $r$, $i$, $j$ which are not
defined mod $q$. We distinguish four cases.

\noindent (i) Suppose $q \mid z$. Then (\ref{cong2}) gives no
condition at $q$ and (\ref{cong3}) becomes
\begin{equation} \label{c3z}
 aj+A(di) \equiv u (s-1).
\end{equation}
Also, $r$ is not determined mod $q$. 

If $q \mid \gamma$, so $q \in P$, then $a \not \equiv 0$ and $A(di)
\equiv 0$ for any choice of $i$ mod $q$. We can choose $s$ arbitrarily
and solve (\ref{c3z}) for $j$. Thus the existence of solutions $i$,
$j$ mod $q$ imposes no restriction on $s$, and $I_q=1$ for $q \in P$.

If $q \mid \zeta \delta$, so $q \in Q$, then $t \equiv 1$, $a \equiv
0$ and $u \not \equiv 0$, and $A(di) \equiv udi$. Then (\ref{c3z})
becomes $di \equiv s-1$. When $q \mid \delta$, so $q \in Q'$, we
have $i \equiv 0$ by (\ref{i-lin}) so $s \equiv 1$. Thus only one of
the $q-1$ possibilities for $s \in \Z_q^\times$ can occur, and
$I_q=q-1$ for $q \in Q'$. If $q \mid \zeta$, so $q \in Q''$, there is
no restriction on $i$ from (\ref{i-lin}), so we may choose $s$
arbitrarily and solve for $i$. Thus $I_q=1$ for $q \in Q''$.

\noindent (ii) If $q \mid \gcd(\gamma,g)$, so that $q \in R \cup S \cup T$, then
$a \not \equiv 0$ and $A(di) \equiv 0$. We consider two subcases.

\noindent (a) If $t \equiv \kappa$ then $c \equiv \lambda a \not \equiv 0$ and
$S(t,di) \equiv 0$. Thus (\ref{cong2}) and (\ref{cong3}) become 
\begin{equation} \label{t-kap1}
  \lambda a j \equiv v(s-1) -(t-1)r, 
\end{equation} 
\begin{equation} \label{t-kap2}
  \qquad aj \equiv (u-zv)(s-1) +ztr, 
\end{equation}
so that 
$$ \lambda(u-zv)(s-1)+\lambda ztr \equiv \lambda aj \equiv
v(s-1)-(t-1)r. $$
Using $1+\lambda z =k$, this simplifies to
\begin{equation} \label{t-kap}
 (kv-\lambda u)(s-1) \equiv (kt-1)r.
 \end{equation} 
It suffices to determine when (\ref{t-kap2}) and (\ref{t-kap}) have
  solutions. 

If $t \equiv \kappa \not \equiv k^{-1}$, we may choose $s$
arbitrarily, and then $r$ (and $j$) are determined. Thus $I_q=q$. This
accounts for $q \in R \cup S_h'$ with $t \equiv \kappa$, and also for
$q \in S_{h,2}^+$.

If however $t \equiv \kappa \equiv k^{-1}$, so either $q \in
S_{h,1}^-$ or $q \in T$ with $t \equiv -1$, then $v \equiv \mu u$,
and, since $k \mu = \lambda$, (\ref{t-kap}) gives no condition on $r$
and $s$. We may then choose $r$, $s$ arbitrarily and solve
(\ref{t-kap2}) for $j$, so $I_q=1$ in these cases.

\noindent (b) If $t \equiv \kappa k^{-1}$ then $c \equiv 0$ and 
(\ref{cong2}) becomes
\begin{equation} \label{t-kap-kinv}
  v S(t,di) \equiv v(s-1) - (t-1)r,
\end{equation}
while (\ref{cong3}) again simplifies to (\ref{t-kap2}).  

If $t \equiv \kappa k^{-1}\not \equiv 1$, then $S(t,di) \equiv 0$ and
(\ref{t-kap-kinv}) determines $r$ once $s$ is chosen. We can then
solve (\ref{t-kap2}) for $j$. Thus $I_q=q$. This accounts for the $q
\in R \cup S_h'$ with $t \equiv \kappa k^{-1}$ and also for $q \in
S_{h,2}^-$.

If however $t \equiv \kappa k^{-1} \equiv 1$, so $q \in S_{h,1}^+$ or
$q \in T$ with $t \equiv 1$, then $v=0$, so (\ref{t-kap-kinv}) gives
no restriction on $r$ and $s$, and (\ref{t-kap2}) can be solved for
$j$. Hence $I_q=1$. 

\noindent (iii) Let $q \mid g$ and $q \mid \delta$, so $q \in U'$. Then $i
\equiv 0$ by (\ref{i-lin}), and $a \equiv c \equiv 0$.
Thus (\ref{cong2}) and (\ref{cong3}) become
\begin{equation} \label{g-del-cong}
   v(s-1) \equiv (t-1)r, \qquad (u-zv)(s-1) + ztr \equiv 0. 
\end{equation}
We have two possibilities for $t$. If $t \equiv 1$ then 
$v \not \equiv 0$. Then the first congruence of (\ref{g-del-cong})
forces $s \equiv 1$ and the second $r \equiv 0$. If $t \equiv
k^{-1}$, we have $v \not \equiv \mu u$. Eliminating $r$ between the
two congruences of (\ref{g-del-cong}) and simplifying, we obtain
$$ (v - \mu u) (s-1) \equiv 0, $$
so that again $s \equiv 1$ and $r \equiv 0$. Thus, in both cases, we
get $I_q=q(q-1)$.

\noindent (iv) Let $q \mid g$ and $q \mid \zeta$, so $q \in U''$. Then
$a \equiv c \equiv 0$, and (\ref{cong2}) and (\ref{cong3}) become
\begin{equation} \label{g-zet-cong}
   vS(t,di) \equiv v(s-1) + (t-1)r, \qquad A(di) \equiv (u-zv)(s-1) + ztr. 
\end{equation} 
Again, there are two possibilities for $t$. If $t \equiv 1$ then 
$v \not \equiv 0$, $S(t,di) \equiv di$ and 
$A(di) \equiv vzdi/(k-1)$. Thus we have 
$$ vdi \equiv v(s-1), \qquad \frac{vzdi}{k-1} \equiv
(u-zv)(s-1)+zr. $$
Given $s$, we may solve the first of these for $i$, and $r$ is 
determined by second. Thus $I_q=q$. If $t \equiv
k^{-1}$ then $v \not \equiv \mu u$, and we have
$$ v(s-1)+(t-1)r \equiv 0, \qquad A(di) \equiv (u-zv)(s-1)+ztr. $$
Again, we may choose $s$ and the first congruence determines
$r$. The second can then be solved for $i$. Thus again $I_q=q$.   
\end{proof}

In order to sum over all quintuples $(t,a,c,u,v) \in \NN_h$, we must
allow for the fact that the partition of $S_h^+$ as $S_{h,1}^+ \cup
S_{h,2}^+$ depends on the choice of quintuple: in choosing a
quintuple, we must in particular choose which primes $q \in
S_h^+$ to allocate to $S_{h,1}^+$, so that $t \equiv 1 \pmod{q}$, and
which to allocate to  $S_{h,2}^+$, so that $t \equiv \kappa k^{-1}
\pmod{q}$. Similarly, we must choose how to allocate the primes $q
\in S_h^-$ between $S_{h,1}^-$ and $S_{h,2}^-$. The remaining sets of
primes listed in Table \ref{Iq}, viz.~$P$, $Q'$, $Q''$, $R$, $S_h'$,
$T$, $U'$, $U''$, are all independent of the choice of quintuple.

Let $I \subseteq S_h^+$ and $J \subseteq S_h^+$. The
number $N_h(I,J)$ of quintuples $(t,a,c,u,v) \in \NN_h$ with
$S_{h,1}^+=I$ and $S_{h,1}^-=J$ is then given by 
$$ N_h(I,J) = \prod_{q \mid e} N_q, $$
where the $N_q$ are as in Table \ref{Iq}, and with $N_q=q(q-1)$
 if $q \in I$ or $q \in J$, and $N_q=q^2(q-1)$ if $q \in S_h^+
\backslash I$ or $q \in S_h^- \backslash J$. For all these quintuples,
$I_h(t,a,c,u,v)$ takes the same value. Denoting this value by  
by $I_h(I,J)$, we may rewrite (\ref{count-braces}) as  
\begin{equation} \label{set-count-braces}
  b(M,A) = \frac{\varphi(\delta)}{\gamma \varphi(e) w} \;
\sum_{h=1}^w \sum_{I,J} \frac{N_h(I,J)}{I_h(I,J)},
\end{equation}
where 
\begin{eqnarray*} 
\frac{N_h(I,J)}{I_h(I,J)} & = & \prod_{q \mid e} \frac{N_q}{I_q} \\
   & = & \left( \prod_{q \in P} q(q-1) \right)  \left( \prod_{q \in
  Q''} (q-1) \right)  
             \left( \prod_{q \in R \cup S'_h \cup T} 2q(q-1) \right) \\
     & &   \quad \times \left( \prod_{q \in  S_h^+ \cup S_h^-} q(q-1) \right)   
                  \left( \prod_{q \in U'} 2 \right)
                 \left( \prod_{q \in U''} 2(q-1) \right)    \\ 
     & = &    \left( \prod_{q \in P \cup R \cup S \cup T} q(q-1) \right)
              \left( \prod_{q \in Q'' \cup U''} (q-1) \right)        
              \left( \prod_{q \in R \cup Sh' \cup T \cup U} 2 \right)   \\  
    & = &  \left( \prod_{q \mid \gamma} q(q-1) \right)
              \left( \prod_{q \mid \gcd(\zeta,e)} (q-1) \right)        
              \left[\left( \prod_{q \mid g} 2 \right) \left( \prod_{q
                  \in S_h^+ \cup S_h'} \right) \right]   \\   
    & = & (\gamma \varphi(\gamma)) \varphi(\gcd(\zeta,e)) \left[
                2^{\omega(g)} 2^{-|S_h^+ \cup S_h^-|} \right]. 
\end{eqnarray*}
This expression is independent of $I$ and $J$ since $N_q/I_q = q(q-1)$
for all $q \in S_h^+ \cup S_h^-$. As there are $2^{|S_h^+|}$ possible
sets $I$, and $2^{|S_h^-|}$ possible sets $J$, summing over $I$ and
$J$ introduces a factor $2^{|S_h^+\cup S_h^-|}$. Thus we have
\begin{eqnarray*}
  b(M,A) & = & \frac{\varphi(\delta)}{\gamma \varphi(e) w} \;
\sum_{h=1}^w  \gamma \varphi(\gamma) \varphi(\gcd(\zeta,e))  2^{\omega(g)} \\
     & = & \frac{\varphi(\delta)}{\gamma \varphi(e) w} \cdot w \cdot 
       \gamma \varphi(\gamma) \varphi(\gcd(\zeta,e))  2^{\omega(g)} .
\end{eqnarray*}
Since $\gamma \mid e$, we have
$$ \varphi(e) =\varphi(\gamma) \varphi(\gcd(\zeta,e))
\varphi(\gcd(\delta),e), $$ 
so the previous expression simplifies to
\begin{eqnarray*}
 b(M,A) & = & \frac{2^{\omega(g)} \varphi(\delta)}{\varphi(\gcd(\delta,e))} \\
    & =  & \frac{2^{\omega(g)} \varphi(\gcd(\delta,e)
   \varphi(\gcd(\delta),d)}{\varphi(\gcd(\delta,e))} \\ 
  & = & 2^{\omega(g)} w.
\end{eqnarray*} 
This completes the proof of Theorem \ref{thm-braces}.

\section{An example where $n$ has $7$ prime factors}

In \cite[\S8]{AB-galois}, we considered four groups of 
order 
$$ n = 2 \cdot 3 \cdot 7 \cdot 43 \cdot 127 \cdot 211 \cdot 337
    = 16\;309\;243\;734,  $$
each with 
$$ g= 43 \cdot 127 \cdot 211 \cdot 337. $$ 
The values of $d$, $z$ and the $r_q$ for primes $q \mid g$ are as
shown in Table \ref{example-table}. (We omit the values of $k$.)

\begin{table}[ht] 
\centerline{ 
\begin{tabular}{|c|c|c|c|c|c|c|} \hline
 & $r_{43}$ & $r_{127}$ & $r_{211}$ & $r_{337}$ & $d$ & $z$ \\ \hline
 $G_1$ & $2$ & $3$ & $7$ & $21$ & $42$ & $1$\\ \hline
 $G_2$ & $2$ & $3$ & $7$ & $21$ & $42$ & $1$ \\ \hline
 $G_3$ & $42$ & $21$ & $14$ & $7$ & $42$ & $1$ \\ \hline
 $G_4$ & $2$ & $7$ & $7$ & $14$ & $14$ & $3$ \\ \hline
\end{tabular}
}  
\vskip3mm

\caption{Parameters for some groups of order $n$.} 
 \label{example-table}  	
\end{table}

No two of these groups are isomorphic. Although the parameters shown
in Table \ref{example-table} are the same for $G_1$ and $G_2$, the
(omitted) values of $k$ are different and do not generate the same
subgroup of $\Z_e^\times$.  In \cite{AB-galois}, we took $\Gamma=G_1$
and, for $1 \leq i \leq 4$, we calculated the number $e(\Gamma,G_i)$ of
Hopf-Galois structures of type $G_i$ on a Galois extension with Galois
group isomorphic to $\Gamma$. We obtained a different answer for each
$i$. In particular $e(G_1,G_1) \neq e(G_1,G_2)$, since the formula in \cite[Theorem 2.2]{AB-galois} for
the number of Hopf-Galois structures depends on the sets $S_h=S_h^+
\cup S_h^-$ and not just on the $r_q$.

When we calculate the number of skew braces $b(G_i,G_j)$, the
nonisomorphic groups $G_1$, $G_2$, $G_3$ behave identically, since the
formula in Theorem \ref{thm-braces} of this paper only depends on the
parameters $g$, $d$, $\gamma$ and $\delta$. In all these cases we have
$\gamma=g$, so that $\gamma \mid e$. If $i$, $j \in \{1,2,3\}$, we
have $d=\delta=42$ and $w=12$, whereas if $i=4$ or $j=4$ then
$\gcd(\delta,d)=14$ and $w=6$.  We can then read off from Theorem
\ref{thm-braces} that
$$  b(G_i,G_j) = \begin{cases} 192 & \mbox{ if } i, j \in \{1,2,3\}, \\
           96 & \mbox{ if } i=4 \mbox{ or } j=4. \end{cases} $$

%

\section{Special Cases}

In this final section, we use Theorem \ref{thm-braces} to calculate
$b(M,A)$ for the various special cases considered in \cite[\S
  9]{AB-galois}. Throughout this section, $A=G(d,e,k)$ and
$M=G(\delta, \epsilon, \kappa)$ are two groups of squarefree order
$n=de=\delta \epsilon$, and $g$, $z,$ $\gamma$, $\zeta$, $w$ are as
defined before the statement of Theorem \ref{thm-braces}.

\subsection{When $M$ or $A$ is cyclic or dihedral}  \label{Gam-cyc}

\begin{corollary} \ 

\begin{itemize}
\item[(i)] If $M$ is cyclic and $A$ is an arbitrary group of order $n$, then
$b(M,A)=2^{\omega(g)}$.

\item[(ii)] If $A$ is cyclic and $M$ is an arbitrary group of order $n$, then 
$b(M,A)= 1$.

\item[(iii)] If $M$ is dihedral of order $n=2m$ with $m$ odd and squarefree,
  and $A$ is an arbitrary group of order $n$, then 
$$ b(M,A) = \begin{cases} 2^{\omega(g)} & \mbox{ if } d=1 \mbox{ or } 2, \\
      0 & \mbox{ if } d>2 \end{cases}.  $$ 

\item[(iv)] If $A$ is dihedral of order $n=2m$ with $m$ odd and squarefree,
  and $M$ is an arbitrary group of order $n$, then
  $b(M,A)=2^{\omega(m)}$.
\end{itemize} 
\end{corollary}
\begin{proof}

\noindent (i) We have $\gamma=1$ so $w=1$ and $b(M,A)=2^{\omega(g)} w
  =2^{\omega(g)}$. 

\noindent (ii) We have $d=g=1$, and $e=n$ so $\gamma \mid e$. Again, $w=1$. Thus
  $b(M,A)=1$. 

\noindent (iii) We have $\delta=2$ and $\gamma=m$. If $d>2$ then $\gamma
\nmid e$, so $b(M,A)=0$. If $d \leq 2$ then $w=1$ and $b(M,A)=2^{\omega(g)}$. 

\noindent (iv) We have $d=2$, $g=m$ and $z=1$. As $2 \nmid \gamma$
(cf.~\cite[Remark 6.1]{AB-galois}), we 
necessarily have $\gamma \mid e$. Again $w=1$, so
$b(M,A)=2^{\omega(m)}$. 
\end{proof}

\begin{remark}
In \cite{Rump-cyc}, Rump determines all braces whose multiplicative
group is a finite cyclic group (not necessarily of squarefree
order). Since he treats only braces, not skew braces, the additive
group is always abelian. As any abelian group of squarefree order is
necessarily cyclic, the only case covered by both Rump and our result
is when $M$ and $A$ are both cyclic of squarefree order, so
$b(M,A)=1$.
\end{remark} 

\subsection{When $n$ the product of two primes}

Let $n=pq$ for prime numbers $p>q$. If $p \not \equiv 1 \pmod{q}$ then any
group of order $pq$ is cyclic and $b(M,A)=1$ by Corollary
\ref{Gam-cyc}(i).  We therefore suppose that $p \equiv 1
\pmod{q}$. There are then two groups of order $n$, the cyclic group
$C_n$ (for which $g=d=1$ and $z=pq$) and the nonabelian group $C_p
\rtimes C_q$ (for which $g=p$, $d=q$, $z=1$). We easily obtain the
values of $b(M,A)$ from Theorem \ref{thm-braces}.

\begin{table}[ht] 
\centerline{ 
\begin{tabular}{|c|c|c|} \hline
         & $A=C_n$ & $A=C_p \rtimes C_q$ \\ \hline
 $M=C_n$ &  $1$ & $2$   \\ \hline
 $M=C_p \rtimes C_q$ &  $1$  & $2(q-1)$   \\ \hline
\end{tabular}
}  
\vskip3mm

\caption{Skew braces for two primes.} 
 \label{braces-2p}  	
\end{table}

\begin{corollary} \label{pq-cor}
If $n=pq$ where $p$, $q$ are primes with $p \equiv 1 \pmod{q}$, and
$M$, $A$ are groups of order $pq$, then the number $b(M,A)$ of skew
braces with multiplicative group $M$ and additive group $A$ is as
shown in Table \ref{braces-2p}. In particular, there are in total
$2q+2$ skew braces of order $pq$.
\end{corollary}

\begin{remark}
For $n$ as in Corollary \ref{pq-cor} with $n \leq 30$, namely $n=6$,
$10$, $14$, $21$, $22$, $26$, the total $2q+2$ agrees with \cite[Table
  5.1]{GV}.
\end{remark}

\begin{table}[b] 
\centerline{ 
\begin{tabular}{|c|ccc|c|c|} \hline
 Factorisation &  $d$ & $g$ & $z$ & Condition & \# groups  \\ \hline
 $1$  & $1$ & $1$ & $p_1 p_2 p_3$ & & $1$  \\
 $2$ &  $p_1$ & $p_2$ & $p_3$ & $p_2 \equiv 1 \pmod{p_1}$ & $1$  \\
 $3$ & $p_1$ & $p_3$ & $p_2$ & $p_3 \equiv 1 \pmod{p_1}$ & $1$  \\
 $4$ & $p_1$ & $p_2 p_3$ & $1$ & $p_2 \equiv p_3 \equiv 1 \pmod{p_1}$ &
 $p_1-1$  \\ 
 $5$ & $p_2$ & $p_3$ & $p_1$ & $p_3 \equiv 1 \pmod{p_2}$ & $1$  \\ 
 $6$ & $p_1 p_2$ & $p_3$ & $1$ & $p_3  \equiv 1 \pmod{p_1 p_2}$ &
 $1$  \\  \hline
\end{tabular}
}  
\vskip5mm

\caption{Isomorphism types for groups of order
  $n=p_1 p_2 p_3$.}  \label{3-prime-table} 
\end{table}   


\begin{sidewaystable} 
\centerline{ 
\begin{tabular}{|c|c|c|c|c|c|c|} \hline
$\downarrow M \quad A \rightarrow$  &
  $1$ & $2$ & $3$  & $4$ & $5$ & $6$  \\ \hline 
$1$
   & $1$ & $2$ & $2$ & $4$ & $2$ & $2$  \\ \hline
$2$
   & $1$ & $2(p_1-1)$ & $2(p_1-1)$ & $4(p_1-1)$ & $0$ & $0$  \\ \hline
$3$
   & $1$ & $2(p_1-1)$ & $2(p_1-1)$ &
  $4(p_1-1)$ & $2$ & $2(p_1-1)$ \\ \hline
$4$ 
   & $1$ & $2(p_1-1)$ & $2(p_1-1)$ & $4(p_1-1)$ &
   $0$ & $0$  \\ \hline 
$5$
   & $1$ & $2$ & $2$ & $4$ & $2(p_2-1)$ & $2(p_2-1)$  \\
  \hline 
$6$
   & $1$ & $2(p_1-1)$ & $2(p_1-1)$ & $4(p_1-1)$ &
  $2(p_2-1)$ & $2(p_1-1)(p_2-1)$   \\  \hline
\end{tabular}
}  
\vskip5mm

\caption{Numbers of skew braces for
  $n=p_1 p_2 p_3$.}  \label{braces} 
\end{sidewaystable}   

\subsection{When $n$ is the product of three primes}

Let $n=p_1 p_2 p_3$ where $p_1<p_2<p_3$ are primes. Subject to certain
congruence conditions between $p_1$, $p_2$ and $p_3$, there are $6$
possible factorisations $n=dgz$ which give rise to groups of order
$n$. We label these factorisations $1$--$6$ as in Table
\ref{3-prime-table}.  The last column shows the number of isomorphism
types of group for each factorisation, as explained in
\cite[\S9.3]{AB-galois}.

Applying Theorem \ref{thm-braces} to each combination of $M$ and $A$,
we obtain the following result:

\begin{theorem} \label{3p-thm}
Let $n=p_1 p_2 p_3$, where $p_1$, $p_2$, $p_3$ are primes satisfying
the conditions $p_i \equiv 1 \pmod{p_j}$ for $i>j$. Let $M$ and $A$ be
  groups of order $n$. Then the number $b(M,A)$ of skew braces with
  multiplicative group $M$ and additive group $A$ is as shown in 
Table \ref{braces}, where the rows (respectively, columns) correspond
to the factorisations of $n$ giving rise to $M$ (respectively $A$) as
in Table \ref{3-prime-table}.
\end{theorem}

\begin{remark}
One can easily obtain analogous results when the congruences $p_i
\equiv 1 \pmod{p_j}$ do not all hold, simply by omitting from Table
\ref{braces} the rows and columns for which the necessary congruences
(as shown in Table \ref{3-prime-table}) are not satisfied.
\end{remark}
 
\begin{corollary}
Let $n=p_1 p_2 p_3$ be as in Theorem \ref{3p-thm}.
\begin{itemize}
\item[(i)] For any group $M$ of order $n$, the number $b(M,\cdot)$ of
  skew braces (up to isomorphism) with multiplicative group $M$ is as
  shown in Table \ref{3p-mul}.
\item[(ii)] For any group $A$ of order $n$, the number $b(\cdot,A)$ of
  skew braces (up to isomorphism) with additive group $A$ is as
  shown in Table \ref{3p-add}.
\item[(iii)] The total number of skew braces of order $n$ (up to
  isomorphism) is 
$$ 4p_1^3 + 4p_1^2 +2p_1p_2 + p_1 + 4p_2 + 4. $$
\end{itemize}
(The rows in Tables \ref{3p-mul}, \ref{3p-add} correspond the
factorisations of $n$ giving $M$ or $A$, as in Table \ref{3-prime-table}.)
\end{corollary}

\begin{table}[b] 
\centerline{ 
\begin{tabular}{|c|c|} \hline
Factorisation for $M$ & $b(M,\cdot)$ \\ \hline 
$1$ &  $4p_1+5$ \\ \hline
$2$ & $4p_1^2-4p_1+1$ \\ \hline
$3$ & $4p_1^2-2p_1+1$ \\ \hline
$4$ & $4p_1^2-4p_1+1$ \\ \hline 
$5$ & $4p_1+4p_2-3$ \\ \hline 
$6$ & $4p_1^2+2p_1p_2-6p_1+1$  \\ \hline  
\end{tabular}
}  
\vskip5mm

\caption{Number of skew braces with multiplicative
  group $M$.}  \label{3p-mul} 
\end{table} 

\bigskip

\begin{table}[b]
\centerline{
\begin{tabular}{|c|c|} \hline
Factorisation for $A$ & $b(\cdot,A)$ \\ \hline 
$1$ & $p_1+4$  \\ \hline
$2$ & $2p_1^2+2p_1$  \\ \hline
$3$ & $2p_1^2+2p_1$  \\ \hline
$4$ & $4p_1^2+4p_1$  \\ \hline 
$5$ & $4p_2$  \\ \hline 
$6$ & $2p_1 p_2$ \\ \hline 
\end{tabular}
}  
\vskip5mm

\caption{Number of skew braces with additive
  group $A$.}  \label{3p-add} 
\end{table}   
\begin{proof}
For each factorisation for $M$, the values of $b(M,\cdot)$ are
obtained by adding the entries $b(M,A)$ in the corresponding row of
Table \ref{braces}, where the entry for $A$ with Factorisation $4$ is
multiplied by $p_1-1$ since there are $p_1-1$ isomorphism types. The
values of $b(\cdot,A)$ are obtained similarly from the columns of
Table \ref{braces}. The total number of skew braces of order $n$ is
obtained by adding the values of $b(M,\cdot)$, or of $b(\cdot,A)$,
again with the value for Factorisation $4$ multiplied by $p_1-1$.
\end{proof}

\bibliography{GSQ-bib}
\end{document}